\documentclass[12pt,letterpaper,reqno]{amsart}

\usepackage{comment}
\usepackage{appendix}
\usepackage{graphicx} 
\usepackage{amsfonts}
\usepackage{amsmath}
\allowdisplaybreaks
\usepackage{amsthm}
\usepackage{xcolor}
\usepackage{mathtools}
\usepackage{fancyhdr,amssymb}
\usepackage{url}
\usepackage{enumerate}

\newcommand{\prob}[1]{\mathbb{P} \left( #1 \right)}

\newcommand{\norm}[1]{\lVert #1 \rVert}

\newcommand{\ceil}[1]{\left\lceil #1 \right\rceil}

\newcommand{\supp}{\mathrm{supp}}

\renewcommand{\hat}{\widehat}
\renewcommand{\tilde}{\widetilde}

\newcommand{\N}{\mathbb{N}}
\newcommand{\Z}{\mathbb{Z}}

\newcommand{\R}{\mathbb{R}}
\newcommand{\C}{\mathbb{C}}

\theoremstyle{plain} 
\newtheorem{theorem}{Theorem}[section]
\newtheorem{proposition}[theorem]{Proposition}
\newtheorem{corollary}[theorem]{Corollary}
\newtheorem{lemma}[theorem]{Lemma}
\newtheorem{conjecture}[theorem]{Conjecture}

\theoremstyle{definition} 
\newtheorem{definition}[theorem]{Definition}

\newtheorem{example}[theorem]{Example}

\theoremstyle{remark} 
\newtheorem{remark}[theorem]{Remark}

\makeatletter
\g@addto@macro\endtheorem{\par\vspace{\baselineskip}}
\g@addto@macro\endproposition{\par\vspace{\baselineskip}}
\g@addto@macro\endcorollary{\par\vspace{\baselineskip}}
\g@addto@macro\endlemma{\par\vspace{\baselineskip}}
\g@addto@macro\enddefinition{\par\vspace{\baselineskip}}
\g@addto@macro\endexample{\par\vspace{\baselineskip}}
\g@addto@macro\endconjecture{\par\vspace{\baselineskip}}
\makeatother

\title{Signal Recovery Using Gabor Frames}
\author{Ivan Bortnovskyi}\email{ib538@cam.ac.uk}\address{Department of Pure Mathematics and Mathematical Statistics, Univeristy of Cambridge, Cambridge, United Kingdom, CB3 0WA}

\author{June Duvivier}\email{june@duvivier.us}\address{Department of Mathematics, Reed College, Portland, OR 97202}

\author{Xiaoyao Huang}\email{xyrushac@umich.edu}\address{Department of Mathematics, University of Michigan, Ann Arbor, 48109}

\author{Alex Iosevich}\email{iosevich@math.rochester.edu}\address{Department of Mathematics, University of Rochester
915 Hylan Building
P.O. Box 270138
Rochester, NY 14627}

\author{Say-Yeon Kwon}\email{sk9017@princeton.edu}\address{Department of Mathematics, Princeton University, Princeton, NJ 08544}

\author{Meiling Laurence}\email{meiling.laurence@yale.edu}\address{Department of Mathematics, Yale University, New Haven, CT 06511}

\author{Michael Lucas}\email{ml2130@cam.ac.uk}\address{Department of Pure Mathematics and Mathematical Statistics, Univeristy of Cambridge, Cambridge, United Kingdom, CB3 0WA}

\author{Steven J. Miller}\email{sjm1@williams.edu}\address{Department of Mathematics, Williams College, Williamstown, MA, 01267}

\author{Tiancheng Pan}\email{tp545@cam.ac.uk}\address{Department of Pure Mathematics and Mathematical Statistics, Univeristy of Cambridge, Cambridge, United Kingdom, CB3 0WA}

\author{Eyvindur Palsson}\email{palsson@vt.edu}\address{Department of Mathematics, Virginia Tech, Blacksburg, VA 24061}

\author{Jennifer Smucker}\email{jennifer21@vt.edu}\address{Department of Mathematics, Virginia Tech, Blacksburg, VA 24061}

\author{Iana Vranesko}\email{yv2@williams.edu}\address{Department of Mathematics, Williams College, Williamstown, MA, 01267}

\begin{document}

\begin{abstract}

We present a novel probabilistic framework for the recovery of discrete signals with missing data, extending classical Fourier-based methods. While prior results, such as those of Donoho and Stark (\cite{DonohoStark}; see also Logan's method in \cite{Logan1965}), guarantee exact recovery under strict deterministic sparsity constraints, they do not account for stochastic patterns of data loss. Our approach combines a row-wise Gabor transform with a probabilistic model for missing frequencies, establishing near-certain recovery when losses occur randomly.

The key innovation is a maximal row-support criterion that allows unique reconstruction with high probability, even when the overall signal support significantly exceeds classical bounds. Specifically, we show that if missing frequencies are independently distributed according to a binomial law, the probability of exact recovery converges to $1$ as the signal size grows. This provides, to our knowledge, the first rigorous probabilistic recovery guarantee exploiting row-wise signal structure.

Our framework offers new insights into the interplay between sparsity, transform structure, and stochastic loss, with immediate implications for communications, imaging, and data compression. It also opens avenues for future research, including extensions to higher-dimensional signals, adaptive transforms, and more general probabilistic loss models, potentially enabling even more robust recovery guarantees.

\end{abstract}


\maketitle

\section{Introduction}
\subsection{Background}
The purpose of this paper is to examine a novel approach to the transmission and recovery of discrete signals. The primary object we consider is the discrete signal
\begin{equation*}
    f:\Z_N\times\Z_T \to \C,
\end{equation*} where $\Z_n$ denotes the integers modulo $n$. Suppose that the set 
\begin{equation*}
    \{f(x) : x \in M\}
\end{equation*}
is missing, where $M \subset \Z_N\times\Z_T$. The question we are interested in is under what conditions can we recover the original signal? One of the main tools that has been used to explore this question is the discrete Fourier transform, defined as follows.
\begin{definition}[Discrete Fourier Transform]
Let $f:\Z_N\times \Z_T \to \C$. The discrete Fourier transform, $\hat f: \Z_N \times \Z_T \to \C$, of $f$ is given by
\begin{equation}
    \hat f(m,n) := \frac{1}{\sqrt{NT}}\sum_{x \in \Z_N}\sum_{y \in \Z_T}f(x,y)\text{exp}\left(-2\pi i\left(\frac{xm}{N}+\frac{yn}{T}\right)\right)
\end{equation}
\end{definition}
Donoho and Stark \cite{DonohoStark} showed that for signals in $\Z_N$, if both the signal’s support and the set of missing frequencies are sufficiently small, then the signal can be uniquely recovered \cite[Theorem 1.1]{burstein2025fourier}.
Although their original result was stated in one dimension, for the purpose of illustrating the main ideas of this paper, we will present it in two dimensions using the two-dimensional discrete Fourier transform. By following the same arguments as in \cite{DonohoStark}, one obtains the following result:
\begin{theorem}\label{Donoho Stark 1}
Let $f:\Z_N\times\Z_T \to \C$, and suppose we transmit the frequencies $\hat{f}$, but the values of $\hat{f}$ are missing in $M \subset \Z_N\times\Z_T$. If $f$ is supported in $E \subset \Z_N\times\Z_T$ and 
\begin{equation}
   \displaystyle|E||M|< \frac{NT}{2}, 
\end{equation} then $f$ can be recovered exactly using Logan’s method \cite{Logan1965}, which consists of finding $f = \arg\min_{g} \|g\|_{L^1(\Z_N \times \Z_T)}$ subject to $\hat{g}(m) = \hat{f}(m)$ for all $m \notin M$.

\end{theorem}
The Fourier inversion theorem gives us the following equivalent version of Theorem \ref{Donoho Stark 1}.
\begin{theorem}\label{Donoho Stark 2}
  Let $f:\Z_N\times\Z_T \to \C$, and suppose we transmit $f$, but the values of $f$ are missing in $M \subset \Z_N\times\Z_T$. If $\hat f$ is supported in $E \subset \Z_N\times\Z_T$ and 
\begin{equation}
   \displaystyle|E||M|< \frac{NT}{2}, 
\end{equation} then $f$ can be recovered exactly using Logan’s method \cite{Logan1965}, which consists of finding $f = \arg\min_g\norm{\hat g}_{L^1(\Z_N\times\Z_T)}$ subject to $f(m) = g(m)$ for all $m \notin M$.
\end{theorem}

Specifically, Theorems \ref{Donoho Stark 1} and \ref{Donoho Stark 2} are both derived from the two-dimensional classical uncertainty principle, which can be easily generalized from the one-dimensional uncertainty principle stated by Donoho and Stark \cite{DonohoStark}.

\begin{theorem}
Let $f:\Z_N \times \Z_T \to \C$ be nonzero. Then,
\begin{align}
    |\supp(f)||\supp(\widehat{f})| \geq NT.
\end{align}
\end{theorem}

We use Theorems \ref{Donoho Stark 1} and \ref{Donoho Stark 2} together to achieve unique recovery with high probability, even when the probability of losing values is much higher than what the classical recovery condition allows. Applying the theorems in conjunction requires a new framework, known as the Gabor transform, which we borrow from continuous harmonic analysis and adapt to our discrete setting. First, we define the continuous Gabor transform. See, for example, \cite{grochenig2001}, for a thorough description of Gabor transforms and their properties. 

\begin{definition}[Continuous Gabor Transform]
Let $f:\R \to \C$ be (Lebesgue) integrable. The Gabor transform of $f$ is defined by
$$
G_f(\omega, \tau) = \int_{-\infty}^{\infty} f(t) g(t-\tau)e^{-i\omega t}dt,
$$
where $g$ is a window function, commonly taken to be $\displaystyle g(t) =\frac{1}{\sqrt{2\pi \sigma^2}}e^{-\frac{t^2}{2\sigma^2}}$.    
\end{definition}

The window function serves to isolate the function within a short time span, so that we are effectively sending the Fourier transform of the restriction of $f$ to this time span. This is advantageous in that it can be easier to reconstruct $f$ with reasonable accuracy based on knowledge of its approximate frequencies in short periods of time than based on an imperfect transmission of the Fourier transform of the entire function.

In the discrete setting, rather than taking the window function to be the normal distribution, we can take the window function to restrict $f$ to particular rows or columns in its domain, which we do in the following definitions. See \cite{IKLM2021} for some example of Gabor transforms in the finite setting. 

\begin{definition}[Gabor Transform Along a Row]
Given a function $f: \Z_N \times \Z_T \to \mathbb{C}$, we define its row-wise Gabor Transform $Gf: \mathbb{Z}_N\times \mathbb{Z}_T \to \mathbb{C}$ by 
\begin{align*}
    Gf(m,a) := N^{-1/2}\sum_{t \in \mathbb{Z}_N} f(t,a) e^{-\frac{2\pi im\cdot t}{N}}, 
\end{align*}
i.e., $Gf(m,a) := \widehat{f(\text{---},a)}(m).$
\\
We thus have the inverse Gabor transform given by
\begin{align*}
    f(t,a) = N^{-1/2}\sum_{m \in \mathbb{Z}_N}Gf(m,a) e^{\frac{2\pi it\cdot m}{N}}.
\end{align*}
\end{definition}

The definition for the Gabor transform along a column is similar.

\begin{definition}[Gabor Transform Along a Column]
Given a function $f: \Z_N \times \Z_T \to \mathbb{C}$, we define its column-wise Gabor Transform $\tilde{Gf}: \mathbb{Z}_N\times \mathbb{Z}_T \to \mathbb{C}$ by 
\begin{align*}
    \tilde{Gf}(t,n) := N^{-1/2}\sum_{a \in \mathbb{Z}_T} f(t,a) e^{-\frac{2\pi in\cdot a}{T}}, 
\end{align*}
i.e., $\tilde{Gf}(t,n) := \widehat{f(t,\text{---})}(n).$
\\
We thus have the inverse Gabor transform given by
\begin{align*}
    f(t,a) = N^{-1/2}\sum_{n \in \mathbb{Z}_T}\tilde{Gf}(t,n) e^{\frac{2\pi ia\cdot n}{T}}.
\end{align*}    
\end{definition}

\subsection{New Results}

Now that we have established the theoretical framework, we can state our main results. In the following theorem, we strengthen the classical unique recovery condition for the case where frequencies are intercepted according to a binomial model, by transmitting row-wise Gabor transforms instead of the Fourier transform of the function $f: \Z_N \times \Z_T \to \C.$ See, for example, \cite{burstein2025fourier}, \cite{CRT06}, and \cite{IM2025} for related example of probabilistic ideas in signal recovery. 

\begin{theorem}\label{main theorem}
    Suppose $f: \Z_N \times \Z_T \to \C$, where $T: \N \to \N$ such that $T(N) = o\left(\sqrt{N}e^N\right)$\footnote{We use standard asymptotic notation: $f(N) = O(g(N))$ means $|f(N)| \leq C|g(N)|$ for some constant $C > 0$ and sufficiently large $N$, while $f(N) = o(g(N))$ means $\lim_{N \to \infty} f(N)/g(N) = 0$.}, and define 
    \begin{align*}
        E_{\max} := \max_{a \in \Z_T}|\text{supp}_t(f(t,a))|.
    \end{align*}
    Suppose we transmit $Gf(m,a)$ for all $(m,a) \in \Z_N \times \Z_T$ and that the distribution of lost frequencies is binomial with fixed probability $\displaystyle 0< \theta < \frac{1}{2E_{\max}}$. Let $M$ be the set of missing frequencies and define 
    \begin{align*}
     &M_{\max}:=\operatorname{max}_{a \in \Z_T}|M\cap \{Gf(t,a): t \in \Z_N\}|, 
     \\&M_{\min}:=\operatorname{min}_{a \in \Z_T}|M\cap \{Gf(t,a): t \in \Z_N\}|.
    \end{align*}
As $N \to \infty$, $\displaystyle \mathbb{P}\left( M_{\max} < \frac{N}{2E_{\max}} \right) \to 1$, which implies that the probability of unique recovery converges to $1$.\\
Furthermore, for $\displaystyle\theta > \frac{1}{2E_{\max}}$, as $N \to \infty$, $\displaystyle \mathbb{P}\left( M_{\min} < \frac{N}{2E_{\max}} \right) \to 0$.
\end{theorem}

\begin{remark}
The binomial loss model assumes that each transmitted frequency is independently lost with a fixed probability $0 < \theta < 1$. This corresponds to the standard i.i.d.\ Bernoulli erasure model commonly used in signal recovery and information theory. While real-world data transmission often exhibits correlated or bursty loss patterns, the independent loss assumption provides a tractable and analytically convenient framework that captures the essential probabilistic behavior relevant to our asymptotic analysis.

 To contextualize the improvement, consider the following. Suppose $f: \Z_N \times \Z_T \to \C$ has support $E \subset \Z_N \times \Z_T$, and we transmit $\hat{f}$, its Fourier transform in $\Z_N \times \Z_T$. Suppose also that the values of $\hat{f}$ are missing in $M$, and that each element of $M$ is lost independently with probability $\theta$. Then, the expectation of lost frequencies is $NT\theta$. So, according to the classical recovery condition 
    \begin{align*}
     |E||M|<\frac{NT}{2},   
    \end{align*}
    $|E|$ can be as large as 
    \begin{align*}
    \ceil{\frac{NT}{2NT \theta}} -1 = \ceil{\frac{1}{2\theta}}-1    
    \end{align*}
    while guaranteeing unique recovery with high probability for $N$ sufficiently large. In our theorem, we only require 
    \begin{align*}
    E_{\max} < \frac{1}{2\theta},   
    \end{align*}
     which means that $|E|$ can be as large as 
     \begin{align*}
      T \left(\ceil{\frac{1}{2\theta}} - 1\right)    
     \end{align*}
     while nearly guaranteeing unique recovery for $N$ large enough, with the restriction on the growth of $T$ being $T = o\left(\sqrt{N}e^N\right).$ 

For some intuition as to why this is the case, note that in any given $i^{th}$ row, $\Z_N \times\{i\},$ the Fourier transform of $f$ as a function restricted to $\Z_N \times\{i\}\simeq \Z_N$ is given by $Gf(\text{---},i).$ The classical unique recovery condition in one dimension tells us that if $$\displaystyle|\supp(f|_{\Z_N\times\{i\}})||M_i|<\frac{N}{2}, \quad \text{ where } M_i=\{m:Gf(m,i)\text{ is missing}\},$$ then we are guaranteed unique recovery of this $i^{th}$ row. If each element of $M_i$ is lost independently with probability $\theta,$ then the expectation of lost frequencies is $N\theta,$ which means $|\supp(f|_{\Z_N\times\{i\}})|$ can be as large as 
$$\ceil{\frac{N}{2N\theta}}-1 =\ceil{\frac{1}{2\theta}}-1$$
while guaranteeing unique recovery with high probability for $N$ sufficiently large. We want this unique recovery condition to be satisfied in each row, which is equivalent to requiring 
$$E_{max} < \frac{1}{2\theta}.$$ 

Since the support of $f$ on all of $\Z_N\times \Z_T$ can be $T$ times the support allowed for $f$ on each row, this means that by transmitting row-wise Gabor transforms rather than the Fourier transform of $\hat f$ as a function on $\Z_N\times\Z_T$, $|E|$ can be as large as 
\begin{align*}
      T \left(\ceil{\frac{1}{2\theta}} - 1\right)    
\end{align*} 
while meeting the conditions for unique recovery of each row. However, this all relies on not just any given $M_i$ being less than $N/2E_{max}$ with high probability as $N\to \infty,$ but that the condition is satisfied for \textit{all} of the $M_i$'s. The crux of the proof thus lies in showing that, under the hypotheses of the theorem, 
\begin{align*}
    \mathbb{P}\left( M_{\max} < \frac{N}{2E_{\max}} \right)
\end{align*} 

converges to $1$ as $N\to \infty.$

As a consequence, if we would like to transmit a signal \( \mathbb{Z}_M \to \mathbb{C} \) and we know that the frequencies will be intercepted binomially, then to ensure a high probability of recovering the original signal, we should split it into as many rows as possible while keeping the number of entries per row large enough for  
\begin{align*}
    \mathbb{P}\left( M_{\max} < \frac{N}{2E_{\max}} \right)
\end{align*}
to be reasonably close to \(1\), and transmit the row-wise Gabor transforms.
\end{remark}

\begin{remark} An additional appeal of recovering $f$ row by row is that even if some rows are unrecoverable, if a sufficient number of rows can be recovered, we can apply Theorem \ref{Donoho Stark 2} to columns of $f$ and still manage to recover the rest of $f$. However, since 
\begin{align*}
    \lim_{N \to \infty} \prob{M_{\min} < \frac{N}{2E_{\max}}} = 0
\end{align*} to deal with $\displaystyle\theta \geq \frac{1}{2E_{\max}}$, we cannot expect a fixed positive proportion of rows to be recoverable as $N$ grows, unless we have additional information about the proportion of rows whose support is strictly smaller than that of the row with maximal support.

This motivates the assumptions in the following theorem, which—while its applications may be somewhat more limited than Theorem \ref{main theorem}—builds on Theorems \ref{Donoho Stark 1} and \ref{Donoho Stark 2} jointly to address $\displaystyle\theta \geq \frac{1}{2E_{\max}}.$ Although outside the scope of this paper, it may be worthwhile to further examine their combined use in settings where $N$ is fixed, rather than $N \to \infty.$ \end{remark}
\begin{theorem}\label{improvement}
Suppose $f: \Z_N \times \Z_T \to \C$, where $T: \N \to \N$ such that $T(N) = o\left(\sqrt{N}e^N\right),$ define $E_{\max}$ as above, and define 
\begin{equation*}
S_{\max}:= \max_{t \in \Z_N} |\supp_n{(\tilde{Gf}(t,n))}|.
\end{equation*} Suppose we transmit $Gf(m,a)$ for all $(m,a) \in \Z_N \times \Z_T$ and that the distribution of lost frequencies is binomial with probability $\displaystyle\theta = \frac{1}{2E_{\max}}+\delta$ for some $\delta \geq 0$. Suppose that 
\begin{equation}
    \left|\left\{a: |\text{supp}_t(f(t, a))| < {\frac{E_{\max}}{1 + 2\delta E_{\max}}}\right\}\right|>T\left(\frac{2S_{\max}-1}{2S_{\max}}\right).
\end{equation} As $N \to \infty$, $\prob{f\text{ can be recovered}} \to 1$.
\end{theorem}
\section{Proofs of Theorems}

We prove Theorems \ref{main theorem} and \ref{improvement} using the following lemma.
\begin{lemma}\label{prob lemma}
    Let $X\sim B(N, \theta)$. Let $g:\N \to \R^\times$ be any function such that \begin{align*}
    \lim_{N \to \infty}\frac{g(N)}{\sqrt{N}e^N} = 0,
    \end{align*} i.e., $g(N) = o\left(\sqrt{N}e^N\right).$
    For $1 > k > \theta$ fixed,
    \begin{equation*}
\prob{X \geq Nk} = o\left(\frac{1}{g(N)}\right).        
    \end{equation*}
For $0 < k < \theta$ fixed,
\begin{equation*}
    \prob{X\leq Nk} = o\left(\frac{1}{g(N)}\right).
\end{equation*}
\end{lemma}

\begin{proof}[Proof of Lemma \ref{prob lemma}]
Note:    
\begin{align}
\prob{X \geq Nk} &= \sum_{y = \lceil Nk \rceil}^{N}\binom{N}{y}\theta^y(1-\theta)^{N-y}\notag 
\\ &\leq \sum_{n=0}^{\infty}\left(\max_{\lceil Nk \rceil \leq y \leq N-1}\frac{(N-y)\theta}{(y+1)(1-\theta)}\right)^n\binom{N}{\lceil  Nk \rceil}\theta^{\lceil  Nk \rceil}(1 - \theta)^{N - \lceil  Nk \rceil}\tag{because $\displaystyle\left(\frac{(N-y)\theta}{(y+1)(1-\theta)}\right)^n$ is the ratio between consecutive summands}
\\ &= \sum_{n=1}^{\infty} \left( \frac{(N-\lceil  Nk \rceil)\theta}{(\lceil  Nk \rceil+1)(1-\theta)}\right)^n\binom{N}{\lceil  Nk \rceil}\theta^{\lceil  Nk \rceil}(1 - \theta)^{N - \lceil  Nk \rceil}\notag
\\
&= \left( 
\frac{1}{1 - \left(\frac{(N - \lceil  Nk \rceil)\theta}{(\lceil  Nk \rceil + 1)(1 - \theta)}\right)} \right)
\binom{N}{\lceil  Nk \rceil}\theta^{\lceil  Nk \rceil}(1 - \theta)^{N - \lceil  Nk \rceil}.\label{bound}
\end{align}\\

\noindent We consider the limits of 
\[
\frac{1}{1 - \left(\frac{(N - \lceil  Nk \rceil)\theta}{(\lceil  Nk \rceil + 1)(1 - \theta)}\right)}
\quad \text{and} \quad
\binom{N}{\lceil  Nk \rceil}\theta^{\lceil  Nk \rceil}(1 - \theta)^{N - \lceil  Nk \rceil}
\]
separately.\\
First, observe that
\[
\frac{N -  Nk  - 1}{ Nk  + 2} 
\leq \frac{N - \lceil  Nk \rceil}{\lceil  Nk \rceil + 1} 
\leq \frac{N -  Nk }{ Nk  + 1}.
\]
Applying the squeeze theorem then gives
\[
\lim_{N \to \infty} \frac{N - \lceil  Nk \rceil}{\lceil  Nk \rceil + 1} = \frac{1-k}{k}.
\]
Next, note that
\[
\frac{\theta}{1 - \theta} = -1 + \frac{1}{1 - \theta},
\]
which is monotone increasing in $\theta$. Since $\theta <  k $, it follows that
\[
\frac{\theta}{1 - \theta} < \frac{ k }{1 -  k }.
\]
Thus, combining our previous estimates, we obtain
\begin{align*}
    \lim_{N \to \infty} \frac{(N - \lceil  Nk \rceil)\theta}{(\lceil  Nk \rceil + 1)(1 - \theta)} 
    < \frac{k(1-k)}{k(1-k)} = 1,
\end{align*}
which immediately implies
\begin{align}
    \lim_{N \to \infty} \frac{1}{1 - \left(\frac{(N - \lceil Nk\rceil)\theta}{(\lceil Nk\rceil + 1)(1 - \theta)}\right)} = C_{\theta,k}, \label{constant}
\end{align}
where $C_{\theta,k}$ is a finite constant depending only on $\theta$ and $k$. Next, we show that the second term, multiplied by $g(N)$, vanishes in the limit:
\begin{align*}
    \lim_{N \to \infty} g(N) \binom{N}{\lceil Nk\rceil} \theta^{\lceil Nk\rceil} (1 - \theta)^{N - \lceil Nk\rceil} = 0.
\end{align*}
To see this, we apply Stirling's approximation,  $\log(n!)=n\log(n)-n+1/2\log(2\pi n) + O(1/n),$ to get $$\binom{N}{mN} = N\left[-m\log(m) - (1 - m)\log(1-m)\right] + O\left(\frac{1}{N}\right)+\frac{1}{2}\log\left(\frac{1}{2\pi Nm(1-m)}\right)$$ for fixed $0<m<1.$ This gives us
\small
\begin{align*}
&g(N) \binom{N}{\lceil Nk\rceil} \theta^{\lceil Nk\rceil} (1 - \theta)^{N - \lceil Nk\rceil} \\
&= \frac{g(N) \cdot \exp\Bigg(N \Big[ -\frac{\lceil Nk \rceil}{N} \log\Big(\frac{\lceil Nk \rceil}{N}\Big)
- \Big(1-\frac{\lceil Nk \rceil}{N}\Big)\log\Big(1-\frac{\lceil Nk \rceil}{N}\Big) \Big] + O\Big(\frac{1}{N}\Big) \Bigg)}{\sqrt{2 \pi N \frac{\lceil Nk \rceil}{N} \left(1-\frac{\lceil Nk \rceil}{N}\right)}} \\
&\quad \cdot \theta^{\lceil Nk \rceil} (1 - \theta)^{N - \lceil  Nk  \rceil}.
\end{align*}

\begin{align*}
&= \frac{g(N)\cdot\text{exp}\left(N \left[ \left(- k  + O\left(\frac{1}{N}\right)\right)\log\left( k  + O\left(\frac{1}{N}\right)\right)-\left(1-\left( k  + O\left(\frac{1}{N}\right)\right)\right)\log\left(1- k  + O\left(\frac{1}{N}\right)\right)\right] \right)}{\sqrt{2\pi N \left( k  + O\left(\frac{1}{N}\right)\right)\left(1- k  + O\left(\frac{1}{N}\right)\right)}}\\
&\cdot\text{exp}\left(\ceil{ Nk }\log(\theta) + \left( N - \ceil{ Nk }\right) \log(1-\theta)+ O\left(\frac{1}{N}\right)\right)
\\
&=\frac{D_{k,\theta}g(N)}{\sqrt{N}}\cdot\text{exp}\left(N \left[ - k  \log\left( k  + O\left(\frac{1}{N}\right)\right)-\left(1- k  \right)\log\left(1- k  + O\left(\frac{1}{N}\right)\right)\right. \right.
\\
&+ \left.\left.  k \log(\theta) + \left( 1 -  k \right) \log(1-\theta)\right]+O\left(\frac{1}{N}\right)\right)
\end{align*}
where $D_{k,\theta}$ is a constant depending on $k$ and $\theta$. \\
Note that
\begin{align*}
&- k  \log\left( k  \right) 
- \left(1 -  k  \right) \log\left(1 -  k  \right) 
+  k  \log(\theta) 
+ \left(1 -  k \right) \log(1-\theta) \\
&= \left(1 -  k \right) \log\left(\frac{x-1}{k^{-1}-1}\right) 
- \log\left(xk\right),
\end{align*}
where we define $\displaystyle x := \frac{1}{\theta}$.\\
Next, observe that the derivative with respect to $x$ is
\begin{align}
\frac{\partial}{\partial x} \Bigg[ \left(1 -  k \right) \log\left(\frac{x-1}{k^{-1}-1}\right) 
- \log\left(xk\right) \Bigg] 
= \frac{x -  k^{-1}}{ k^{-1} x (1 - x)}, \label{derivative}
\end{align}
which vanishes at $x =  k^{-1}$. \\
Since $k^{-1}>1$, the derivative in \eqref{derivative} is negative for all $x >  k^{-1}$, the function is strictly decreasing in this regime. Consequently, for all $x >  k^{-1}$, we have
\begin{align*}
\left(1 -  k \right) \log\left(\frac{x-1}{k^{-1}-1}\right) 
- \log\left(xk\right)
< \left(1 -  k \right) \log\left(\frac{k^{-1}-1}{k^{-1}-1}\right) 
- \log\left(k^{-1}k\right) = 0.
\end{align*}
Thus, for all $ k > \theta$, we equivalently have:
\begin{align*}
   - k  \log\left( k  \right) 
- \left(1 -  k  \right) \log\left(1 -  k  \right) 
+  k  \log(\theta) 
+ \left(1 -  k \right) \log(1-\theta)< 0.
\end{align*}
Therefore,
\begin{align*}
&\lim_{N \to \infty}  - k  \log\left( k  + O\left(\frac{1}{N}\right)\right)-\left(1- k  \right)\log\left(1- k  + O\left(\frac{1}{N}\right)\right)+  k \log(\theta) + \left( 1 -  k \right) \log(1-\theta)\\
&= - k  \log\left( k  \right) 
- \left(1 -  k  \right) \log\left(1 -  k  \right) 
+  k  \log(\theta) 
+ \left(1 -  k \right) \log(1-\theta)< 0.
\end{align*}
Hence,
\begin{align*}
&\lim_{N \to \infty}\frac{D_{k,\theta}g(N)}{\sqrt{N}}\cdot\text{exp}\left(N \left[ - k  \log\left( k  + O\left(\frac{1}{N}\right)\right)-\left(1- k  \right)\log\left(1- k  + O\left(\frac{1}{N}\right)\right)\right. \right.
\\
&+ \left.\left.  k \log(\theta) + \left( 1 -  k \right) \log(1-\theta)\right]+O\left(\frac{1}{N}\right)\right)
\\
&= \lim_{N \to \infty}\frac{D_{k,\theta}g(N)}{\sqrt{N}}\cdot\text{exp}(-N) = 0,
\end{align*}
so we have shown that 
\begin{align*}
    \lim_{N \to \infty}g(N)\binom{N}{\lceil  Nk \rceil}\theta^{\lceil  Nk \rceil}(1 - \theta)^{N - \lceil  Nk \rceil} = 0,
\end{align*}
i.e., 
\begin{align*}
   \binom{N}{\lceil  Nk \rceil}\theta^{\lceil  Nk \rceil}(1 - \theta)^{N - \lceil  Nk \rceil} = o\left(\frac{1}{g(N)}\right). 
\end{align*}
Thus, returning to [\ref{bound}] and [\ref{constant}], we have
\begin{align*}
    \prob{X \geq Nk} &\leq  \left( 
\frac{1}{1 - \left(\frac{(N - \lceil  Nk \rceil)\theta}{(\lceil  Nk \rceil + 1)(1 - \theta)}\right)} \right)
\binom{N}{\lceil  Nk \rceil}\theta^{\lceil  Nk \rceil}(1 - \theta)^{N - \lceil  Nk \rceil}
\\
&=o\left(\frac{1}{g(N)}\right),
\end{align*}
so we have now proved the first part of the lemma. To prove the latter half of the lemma, we note that if $X \sim B(N,\theta)$, then $N - X \sim B(N, 1 - \theta)$. Applying the result that we have just proved to the random variable $N-X$, we have
\begin{align*}
    \prob{N - X \geq Nk'} = o\left( \frac{1}{g(N)}\right)
\end{align*} for $1 > k' > 1 - \theta$. Now note:
\begin{align*}
 \prob{N - X \geq Nk'} = \prob{X \leq N(1-k')}   
\end{align*} and the condition $1 > k' > 1 - \theta$ is equivalent to $\theta > 1 - k' > 0$. Thus, renaming $k:= 1-k'$, we now have that
\begin{align*}
    \prob{X \leq Nk} = o\left(\frac{1}{g(N)}\right)
\end{align*} for $\theta> k> 0$ and the lemma is proved.
\end{proof}
\begin{proof}[Proof of Theorem \ref{main theorem}]
To begin the proof, we note that 
\begin{align*}
    \mathbb{P}\left( M_{\max} < \frac{N}{2E_{\max}}\right) =\left(\mathbb{P}\left(X < \frac{N}{2E_{\max}}\right)\right)^T = \left(1 - \prob{X \geq\frac{N}{2E_{\max}}}\right)^T
\end{align*}
where $X\sim B(N, \theta)$. 

\noindent By the first part of Lemma \ref{prob lemma}, taking $\displaystyle k = \frac{1}{2E_{\max}}$ and letting $g(N)$ be a function such that $\max\{N, T(N) \}\leq g(N),$ $g(N) = o\left(\sqrt{N}e^N\right),$ we have
\begin{align*}
    \prob{X \geq \frac{N}{2E_{\max}}} =o\left(\frac{1}{g(N)}\right)
\end{align*}
for $\displaystyle\theta < \frac{1}{2E_{\max}}$. We therefore have
\begin{align*}
&\lim_{N \to \infty}\mathbb{P}\left( M_{\max} < \frac{N}{2E_{\max}}\right) = \lim_{N \to \infty}\left(1 - o\left(\frac{1}{g(N)}\right) \right)^{T} \geq \lim_{N \to \infty}\left(1 - o\left(\frac{1}{g(N)}\right) \right)^{g(N)}
\\&= \lim_{g(N) \to \infty}\left(1 - o\left(\frac{1}{g(N)}\right)\right)^{g(N)}=1
\end{align*}
for $\displaystyle\theta < \frac{1}{2E_{\max}}$.
\\\\
Likewise, taking $\displaystyle k = \frac{1}{2E_{\max}}$ again,
\begin{align*}
    \prob{M_{\min}<\frac{N}{2E_{\max}}} &= 1 - \prob{X \geq \frac{N}{2E_{\max}}}^T
    \\ &=1 - \left(1 - \prob{X< \frac{N}{2E_{\max}}} \right)^T
    \\ &\leq 1 - \left(1 - \prob{X\leq \frac{N}{2E_{\max}}} \right)^T
    \\ &= 1 - \left(1 - o\left(\frac{1}{g(N)}\right)\right)^T
\end{align*}
for $\displaystyle\theta > \frac{1}{2E_{\max}}$ by the second part of Lemma \ref{prob lemma}. Therefore,
\begin{align*}
    0&\leq\lim_{N \to \infty}\prob{M_{\min}<\frac{N}{2E_{\max}}} \leq \lim_{N \to \infty}1 - \left(1 - o\left(\frac{1}{g(N)}\right)\right)^T \\&\leq \lim_{N \to \infty}1 - \left(1 - o\left(\frac{1}{g(N)}\right)\right)^{g(N)}= 1-1 = 0
\end{align*}
for $\displaystyle \theta > \frac{1}{2E_{\max}}$. Therefore,
\begin{align*}
\lim_{N \to \infty}\prob{M_{\min}<\frac{N}{2E_{\max}}} = 0    
\end{align*}
for $\displaystyle \theta > \frac{1}{2E_{\max}}$, concluding the proof of the theorem.
\end{proof}
\begin{proof}[Proof of Theorem \ref{improvement}]
Define 
\begin{align*}
    A: &= \left\{a: |\text{supp}_t(f(t, a))| < {\frac{E_{\max}}{1 + 2\delta E_{\max}}}\right\}\\
    &= \left\{a: |\text{supp}_t(f(t, a))| \leq \ceil{\frac{E_{\max}}{1 + 2\delta E_{\max}}} - 1\right\}.   
\end{align*}
We claim that
\begin{equation*}
    \lim_{N \to \infty}\prob{f(\text{---},a)\text{ recovered }\forall a \in A} = 1.
\end{equation*}
Note: 
\begin{align*}
    \prob{f(\text{---},a)\text{ recovered }\forall a \in A} &= \prob{|M \cap \{(t,a): t \in \Z_N\}| < \frac{N}{2|\supp_t(f(t,a))|} \forall a \in A}
    \\
    &\geq \prob{X <\frac{N}{2\max_{a \in A}|\supp_t(f(t,a))|}}^{|A|}
    \\
    &\geq \prob{X <\frac{N}{2\max_{a \in A}|\supp_t(f(t,a))|}}^T
    \\
    &\geq \prob{X<\frac{N}{2\left(\ceil{\frac{E_{\max}}{1 + 2\delta E_{\max}}} - 1\right)}}^T
    \\
    &=\left(1 - \prob{X\geq\frac{N}{2\left(\ceil{\frac{E_{\max}}{1 + 2\delta E_{\max}}} - 1\right)}}\right)^T
\end{align*}
where $X \sim B(N, \theta).$ Since
\begin{equation*}
   \frac{1}{2\left(\ceil{\frac{E_{\max}}{1 + 2\delta E_{\max}}} - 1\right)} > \frac{1}{2\left(\frac{E_{\max}}{1 + 2\delta E_{\max}}\right)} =  \frac{1 + 2\delta E_{\max}}{2E_{\max}} = \theta,
\end{equation*}
by Lemma \ref{prob lemma},
\begin{align*}
\prob{X\geq\frac{N}{2\left(\ceil{\frac{E_{\max}}{1 + 2\delta E_{\max}}} - 1\right)}} = o\left(\frac{1}{g(N)}\right)    
\end{align*}
where we take $g(N)$ to be a function such that $\max\{N, T(N) \}\leq g(N),$ $g(N) = o\left(\sqrt{N}e^N\right).$ Thus,
\begin{align*}
    &\lim_{N \to \infty}\left(1 - \prob{X\geq\frac{N}{2\left(\ceil{\frac{E_{\max}}{1 + 2\delta E_{\max}}} - 1\right)}}\right)^T = \lim_{N \to \infty}\left(1 - o\left(\frac{1}{g(N)}\right)\right)^T \\&\geq \lim_{N \to \infty}\left(1 - o\left(\frac{1}{g(N)}\right)\right)^{g(N)}=\lim_{g(N) \to \infty}\left(1 - o\left(\frac{1}{g(N)}\right)\right)^{g(N)} =1.
\end{align*}
So, 
\begin{equation*}
    \lim_{N \to \infty}\prob{f(\text{---},a)\text{ recovered }\forall a \in A} = 1,
\end{equation*}
which implies
\begin{equation*}
    \lim_{N \to \infty}\prob{\text{more than }T\left(\frac{2S_{\max}-1}{2S_{\max}}\right)\text{ rows recovered}} = 1 \because |A| > T\left(\frac{2S_{\max}-1}{2S_{\max}}\right).
\end{equation*}
Thus, as $N \to \infty$, the probability in each column (i.e., the sets $\{(t,a): a \in \Z_T\}$) of successfully transmitting more than $\displaystyle T\left(\frac{2S_{\max}-1}{2S_{\max}}\right)$ values of $f$ converges to $1$. With probability approaching $1$ as $N \to \infty$, we have in each $t^{th}$ column $\{(t,a): a \in \Z_T\}$, we get the following:
\begin{align*}
    &|\{\text{values of }f(t,a)\text{ missing}: a \in \Z_T\}||\supp_n(\tilde{Gf}(t,n))|\\
    &=(T- |\{\text{values of }f(t,a)\text{ recovered}: a \in \Z_T\}|)|\supp_n(\tilde{Gf}(t,n))|
    \\&<\left(T- T\left(\frac{2S_{\max}-1}{2S_{\max}}\right)\right)|\supp_n(\tilde{Gf}(t,n))|\\
    &=\frac{T|\supp_n(\tilde{Gf}(t,n))|}{2S_{\max}}\\
    &\leq \frac{TS_{\max}}{2S_{\max}}
    \\
    &= \frac{T}{2}.
\end{align*}
Thus, by Theorem \ref{Donoho Stark 2}, the probability of unique recovery of each of the columns converges to $1$ as $N \to \infty$.
\end{proof}

\section{Future Work}
In this paper, the distribution considered for lost frequencies was the binomial distribution. Future work could examine how much improvement results from transmitting row-wise Gabor transforms, rather than the Fourier transform of the function in $ \Z_N \times \Z_T,$ when frequencies are lost according to other distributions. 

Our work also bears similarities to \cite{2025additiveenergyuncertaintyprinciple} in that we obtain an improved unique recovery condition by considering the structure of the support of our function $f$ or its Fourier transform $\widehat{f}.$ An open question is whether or not a new uncertainty principle could arise from considering row- or column-wise support/Fourier support structure. In addition, although we apply the classical recovery condition to our rows and columns for simplicity, our results can easily be made sharper by applying sharper recovery results, such as those found in \cite{2025additiveenergyuncertaintyprinciple}.

It would also be meaningful to consider the scenario in which we can only send part of the signal and seek an optimal, though maybe imperfect, recovery. In this case, transmitting a combinations of row- and column-wise Gabor transforms in a greedy manner on their supports seems to be optimal. To avoid duplicates when sending the intersection of a row or column, we modify the grids: once we send a row or column, we remove it from the previous grid and then obtain a new grid; we always work on the new grid. In this way, we send the part with the most information of the signal.

Moreover, since Theorems \ref{main theorem} and \ref{improvement} both take $N \to \infty,$ it would be natural to explore rates of convergence to determine how best to split a signal $f: \Z_M \to \C$ for a given $M$ into rows to maximize the probability of recovery if we transmit the row-wise Gabor transforms. It may also be worthwhile to explore set-ups similar to that of Theorem \ref{improvement}, where there is a positive probability that some (but not all) rows can be recovered, after which we attempt column-wise recovery, but rather than taking $N \to \infty,$ we try to prove results for general $N.$ For example, given $\delta > 0$, we explore how large must $N$ be for failure probability less than $\delta$, providing engineering guidance.

Lastly, while our paper focuses on row- and column-wise Gabor transforms, we can also consider other choices of window functions, including but not limited to, lines and circles in $\Z_N \times \Z_T.$

\section{Acknowledgements}
This paper was written as part of the SMALL REU Program 2025. We appreciate the support of Williams College, as well as funding from the Dr. Herchel Smith Fellowship Fund, Princeton University, The Finnerty Fund, The Winston Churchill Foundation of the United States, Yale University and NSF grant DMS-2241623. A.I. was supported in part by the National Science Foundation under grants no. HDR TRIPODS-1934962, NSF DMS 2154232, and NSF DMS 2506858 during work on this paper.

\bibliographystyle{alpha}
\bibliography{bibliography}

\end{document}